\documentclass[11pt,a4paper]{article}
\usepackage[cp1251]{inputenc}
\usepackage{amsmath,amsthm}
\usepackage{amsfonts,amstext,amssymb,verbatim,epsfig}
\usepackage{dsfont}

\def\R{{\mathbb{R}}}
\def\N{{\mathbb{N}}}
\def\Z{{\mathbb{Z}}}
\newcommand{\PP}{\mathbb{P}}

\newtheorem{theorem}{Theorem}
\newtheorem{corollary}{Corollary}
\newtheorem{lemma}{Lemma}
\newtheorem{proposition}{Proposition}

\newtheoremstyle{likedef}
  {}%
  {}%
  {}%
  {\parindent}%
  {\bfseries}%
  {.}%
  {.5em}%
  {}%

\theoremstyle{likedef}
\newtheorem{definition}{Definition}[section]
\newtheorem{remark}{Remark}

\numberwithin{equation}{section}

\begin{document}

\title{On the transience of random interlacements}

\author{Bal\'azs R\'ath\thanks{ETH Z\"urich, Department of Mathematics, R\"amistrasse 101, 8092 Z\"urich. Email: balazs.rath@math.ethz.ch and artem.sapozhnikov@math.ethz.ch.
The research of both authors has been supported by the grant ERC-2009-AdG  245728-RWPERCRI.}
\and
Art\"{e}m Sapozhnikov\footnotemark[1]
}

\maketitle


\begin{abstract}
We consider the interlacement Poisson point process on the space of doubly-infinite $\Z^d$-valued trajectories modulo time-shift, tending to infinity at positive and negative infinite times. 
The set of vertices and edges visited by at least one of these trajectories is the graph induced by the random interlacements at level $u$ of Sznitman \cite{SznitmanAM}. 
We prove that for any $u>0$, almost surely, the random interlacement graph is transient. 
\end{abstract}

\section{Introduction}
\noindent

The model of random interlacements was recently introduced by Sznitman in \cite{SznitmanAM}. 
Among other results in \cite{SznitmanAM}, he proved that the random interlacement graph is almost surely connected. 
This result was later refined in \cite{PT} and \cite{RS} by showing that 
every two points of the random interlacement graph are connected via at most $\lceil d/2\rceil$ random walk trajectories, 
and this number is optimal. 
In this paper we further exploit the method of \cite{RS} in order to 
show that the graph induced by the random interlacements is almost surely transient in dimensions $d\geq 3$.

\subsection{The model}\label{sec:model}
\noindent

Let $W$ be the space of doubly-infinite nearest-neighbor trajectories in $\Z^d$ ($d\geq 3$) which tend to infinity at positive and negative infinite times, and 
let $W^*$ be the space of equivalence classes of trajectories in $W$ modulo time-shift. 
We write $\mathcal W$ for the canonical $\sigma$-algebra on $W$ generated by the coordinates $X_n$, $n\in\Z$, and 
$\mathcal W^*$ for the largest $\sigma$-algebra on $W^*$ for which the canonical map $\pi^*$ from $(W,\mathcal W)$ to $(W^*,\mathcal W^*)$ is measurable. 
Let $u$ be a positive number. 
We say that a Poisson point measure $\mu$ on $W^*$ has distribution $\mathrm{Pois}(u,W^*)$ if the following properties hold: 
for a finite subset $A$ of $\Z^d$, denote by $\mu_A$ the restriction of $\mu$ to the set of trajectories from $W^*$ that intersect $A$, and 
by $N_A$ the number of trajectories in $\mathrm{Supp}(\mu_A)$, 
then $\mu_A = \sum_{i=1}^{N_A}\delta_{\pi^*(X_i)}$, where $X_i$ are doubly-infinite trajectories from $W$ parametrized in such a way that $X_i(0) \in A$ and $X_i(t) \notin A$ for all $t<0$ and for all $i\in\{1,\ldots,N_A\}$, and 
\begin{itemize}
\item[(1)]\label{distr:1}
The random variable $N_A$ has Poisson distribution with parameter $u\mathrm{cap}(A)$ (see \eqref{eq:defcap} for the definition of the $\mathrm{cap}(A)$). 
\item[(2)]\label{distr:2}
Given $N_A$, the points $X_i(0)$, $i\in\{1,\ldots,N_A\}$, are independent and distributed according to the normalized equilibrium measure on $A$ (see \eqref{eq:defneqm} for the definition). 
\item[(3)]\label{distr:3}
Given $N_A$ and $(X_i(0))_{i=1}^{N_A}$, the corresponding forward and backward paths are conditionally independent, $(X_i(t), t\geq 0)_{i=1}^{N_A}$ are distributed as independent simple random walks, and $(X_i(t), t\leq 0)_{i=1}^{N_A}$ are distributed as independent random walks conditioned on not hitting $A$.
\end{itemize}
Properties (1)-(3) uniquely define $\mathrm{Pois}(u,W^*)$ as proved in Theorem~1.1 in \cite{SznitmanAM}. 
In fact, Theorem~1.1 in \cite{SznitmanAM} gives a coupling of the Poisson point measures $\mu(u)$ with distribution $\mathrm{Pois}(u,W^*)$ for all $u>0$, but 
we will not need such a general statement here.  
We also mention the following property of the distribution $\mathrm{Pois}(u,W^*)$, which will be useful in the proofs. 
It follows from the above definition of $\mathrm{Pois}(u,W^*)$. 
\begin{enumerate}
\item[(4)]\label{distr:4}
Let $\mu_1$ and $\mu_2$ be independent Poisson point measures on $W^*$ with distributions $\mathrm{Pois}(u_1,W^*)$ and $\mathrm{Pois}(u_2,W^*)$, respectively. 
Then $\mu_1+\mu_2$ has distribution $\mathrm{Pois}(u_1+u_2,W^*)$. 
\end{enumerate}
We refer the reader to \cite{SznitmanAM} for more details. 
For a Poisson point measure $\mu$ with distribution $\mathrm{Pois}(u,W^*)$, 
the {\it random interlacement graph} $\mathcal I = \mathcal I(\mu)$ (at level $u$) 
is defined as the subgraph of $\Z^d$ induced by $\mu$, i.e., 
its vertices and edges are those that are traversed by at least one of the random walks from 
$\mathrm{Supp}(\mu)$. 
It follows from \cite{SznitmanAM} that $\mathcal I$ is a translation invariant ergodic random subgraph of $\Z^d$. 

We consider the simple random walk on the graph $\mathcal{I}$, with uniform edgeweights, i.e., at each step the random walker moves to a uniformly chosen
 neighbor of the current vertex. We say that a graph is transient if the simple random walk on the graph is transient. 
Since $\mathcal I$ is a translation invariant ergodic random subgraph of $\Z^d$, it is transient with probability $0$ or $1$.

\subsection{The result}
\noindent

Our main result is the following theorem. 
\begin{theorem}\label{thm:Resistance}
Let $d\geq 3$ and $u>0$. 
Let $\mu$ be a random point measure on $W^*$ distributed as $\mathrm{Pois}(u,W^*)$, and $\mathbb P$ be the law of $\mu$. 
Then, $\mathbb P$-a.s., the random interlacement graph $\mathcal I = \mathcal I(\mu)$ is transient. 
\end{theorem}
The main ingredient of the proof of Theorem~\ref{thm:Resistance} is Proposition~\ref{prop:Connectivity}.
For $x,y\in\Z^d$ and a positive integer $r$, 
we write $x\stackrel{\mathcal I\cap B(r)}\longleftrightarrow y$ if there is 
a nearest-neighbor path in $\mathcal I$ that connects $x$ and $y$ and uses only vertices of $\mathcal I$ from the $l^\infty$-ball of radius $r$ centered at the origin. 
(In particular, $x$ and $y$ must be vertices in $\mathcal I\cap B(r)$.)

\begin{proposition}\label{prop:Connectivity}
Let $d\geq 3$ and $u>0$. Let $\mathcal I$ be the random interlacement graph at level $u$. 
There exist constants $c = c(d,u)>0$ and $C = C(d,u)<\infty$ such that for all $R\geq 1$, 
\begin{equation}\label{eq:Connectivity}
{\mathbb P}\left[\mathcal I\cap B(R)\neq\emptyset, \bigcap_{x,y\in\mathcal I\cap B(R)}\left\{x\stackrel{\mathcal I\cap B(2R)}\longleftrightarrow y\right\}\right] 
\geq 1 - C\exp\left(-cR^{1/6}\right) .\
\end{equation}
\end{proposition}
\begin{remark}
The exponent $1/6$ is not optimal, but suffices for the proof of Theorem~\ref{thm:Resistance}. 
In fact, Theorem~\ref{thm:Resistance} follows if the probability in \eqref{eq:Connectivity} tends to $1$, as $R\to\infty$, faster than any polynomial. 
\end{remark}
\begin{remark}
Random interlacements were defined on arbitrary transient graphs in \cite{Teixeira}. 
It was proved in \cite{TT11} that for any transient transitive graph $G$,
the interlacement graph on $G$ is almost surely connected for all $u>0$ if and only if $G$ is amenable. 
The following question arises naturally: does Theorem~\ref{thm:Resistance} hold for any transient amenable transitive graph?
\end{remark}

\section{Notation and facts about Green function and capacity}\label{sec:notation}
\noindent

In this section we collect most of the notation, definitions and facts used in the paper. 
For $a\in\R$, we write $|a|$ for the absolute value of $a$, $\lfloor a\rfloor$ for the integer part of $a$, and $\lceil a\rceil$ for the smallest integer not less than $a$. 
For $x\in \Z^d$, we write $|x|$ for $\max\left(|x_1|,\ldots,|x_d|\right)$. For a set $S$, we write $|S|$ for the cardinality of $S$. 
For $R>0$ and $x\in\Z^d$, let $B(x,R) = \{y\in\Z^d~:~|x-y|\leq R\}$  be the $l^\infty$-ball of radius $R$ centered at $x$, and $B(R) = B(0,R)$.  
We denote by $\mathds{1}(A)$ the indicator of event $A$, and by $E[X;A]$ the expected value of random variable $X\mathds{1}(A)$. 

Throughout the paper we always assume that $d\geq 3$. 
For $x\in\Z^d$, let $P_x$ be the law of a simple random walk $X$ on $\Z^d$ with $X(0) = x$. 
We write $g(\cdot,\cdot)$ for the Green function of the walk: for $x,y\in\Z^d$,   
$g(x,y) = \sum_{t=0}^\infty P_x[X(t) = y]$. 
We also write $g(\cdot)$ for $g(0,\cdot)$. The Green function is symmetric and, by translation invariance, $g(x,y) = g(y-x)$.   
It follows from \cite[Theorem~1.5.4]{LawlerRW} that for any $d\geq 3$ there exist a positive constant $c_g=c_g(d)$ and a finite constant $C_g = C_g(d)$ such that 
for all $x$ and $y$ in $\Z^d$, 
\begin{equation}\label{eq:gfbounds}
c_g\min\left(1, |x-y|^{2-d}\right)\leq g(x,y)\leq C_g\min\left(1, |x-y|^{2-d}\right) .\
\end{equation}
\begin{definition}\label{def:cap}
Let $K$ be a subset of $\Z^d$. The energy of a finite Borel measure $\nu$ on $K$ is 
\[
\mathcal E(\nu) = \int_K\int_K g(x,y) d\nu(x) d\nu(y) = \sum_{x,y\in K} g(x,y)\nu(x)\nu(y) .\
\]
The capacity of $K$ is 
\begin{equation}\label{eq:defcap}
\mathrm{cap}(K) = \left[\inf_{\nu} \mathcal E(\nu)\right]^{-1} ,\
\end{equation}
where the infimum is over probability measures $\nu$ on $K$. (We use the convention that $\infty^{-1} = 0$, i.e., $\mathrm{cap}(\emptyset) = 0$.)
\end{definition}
The following properties of the capacity immediately follow from \eqref{eq:defcap}: 
\begin{eqnarray}
\mbox{Monotonicity:}
&&\mbox{for any}~K_1\subset K_2\subset \Z^d, ~~\mathrm{cap}(K_1)\leq \mathrm{cap}(K_2) ;\label{cap:monotonicity}\\
\mbox{Subadditivity:}
&&\mbox{for any}~K_1, K_2\subset \Z^d, ~~\mathrm{cap}(K_1\cup K_2) \leq \mathrm{cap}(K_1)+\mathrm{cap}(K_2) ;\label{cap:subadditivity}\\
\mbox{Capacity of a point:}
&&\mbox{for any}~x\in\Z^d, ~~\mathrm{cap}(\{x\}) = 1/g(0) .\label{cap:point}
\end{eqnarray}
It will be useful to have an alternative definition of the capacity. 
\begin{definition}
Let $K$ be a finite subset of $\Z^d$. The equilibrium measure of $K$ is defined by 
\begin{equation}\label{eq:defeqm}
e_K(x) = P_x\left[X(t)\notin K~\mbox{for all}~t\geq 1\right] \mathds{1}(x\in K),~~x\in\Z^d .\
\end{equation}
The capacity of $K$ is then equal to the total mass of the equilibrium measure of $K$:
\begin{equation}\label{eq:defcap2}
\mathrm{cap}(K) = \sum_x e_K(x) ,\
\end{equation}
and the unique minimizer of the variational problem \eqref{eq:defcap} is given by 
the normalized equilibrium measure 
\begin{equation}\label{eq:defneqm}
\widetilde e_K(x) = e_K(x)/\mathrm{cap}(K) .\
\end{equation}
(See, e.g., Lemma~2.3 in \cite{JainOrey} for a proof of this fact.)
\end{definition}
As a simple corollary of the above definition, we get 
\begin{equation}\label{eq:hittingformula}
P_x\left[H_K < \infty\right] = \sum_{y\in K} g(x,y) e_K(y),~~\mbox{for}~x\in\Z^d .\
\end{equation}
Here, we write $H_K$ for the first entrance time in $K$, i.e., $H_K = \inf\{t\geq 0~:~X(t)\in K\}$. 
We will repeatedly use the following bound on the capacity of $B(0,R)$ in $d\geq 3$ (see (2.16) on page 53 in \cite{LawlerRW}): 
there exist constants $c_b = c_b(d)>0$ and $C_b = C_b(d)<\infty$ such that for all positive $R$, 
\begin{equation}\label{eq:capball}
c_b R^{d-2}\leq \mathrm{cap}\left(B(0,R)\right) \leq C_b R^{d-2} .\
\end{equation}
Finally, we will often use in the proofs the following large deviation bounds for the Poisson distribution, 
which can be proved using the exponential Chebyshev enequality. 
Let $\xi$ be a random variable which has Poisson distribution with parameter $\lambda$, then 
\begin{equation}\label{eq:Poisson}
\mathbf P\left[\lambda/2\leq \xi\leq 2\lambda\right] \geq 1 - 2e^{-\lambda/10} .\
\end{equation}
Throughout the text, we write $c$ and $C$ for small positive and large finite constants, respectively, that may depend on $d$ and $u$. 
Their values may change from place to place.

\section{Proof of Theorem~\ref{thm:Resistance}}
\noindent

We recall the following result about an equivalent characterization of the transience of simple random walk on a graph.
The statement and proof of a more general theorem about an equivalent characterisation of the transience of reversible Markov chains can be found
 on page 398 of \cite{lyons}.
\begin{lemma}\label{l:lyons}
Let $G = (V,E)$ denote a countable, simple graph in which the degree of each vertex is finite. The simple random walk on $G$ is transient
 if and only if there exist real numbers $(u(x,y))_{x,y \in V}$ with the following properties:
\begin{enumerate}
\item[(i)]
$u(y,x) = -u(x,y)$ and $u(x,y) \neq 0$ only if $\{x,y\} \in E$,
\item[(ii)]
$\sum_{x \in V} \left| \sum_{y \in V} u(x,y) \right|<\infty$ and 
$\sum_{x \in V}  \left(\sum_{y \in V} u(x,y)\right) \neq 0$,
\item[(iii)]
$\sum_{x \in V}  \sum_{y \in V} u(x,y)^2 <\infty$.
\end{enumerate}
\end{lemma}
We refer to a function $(u(x,y))_{x,y \in V}$ satisfying $(i)$ as a \emph{flow} on $G$ and $u(x,y)$ as the amount of net flow from vertex $x$ to
vertex $y$.
We say that $\sum_{y \in V} u(x,y)$ is the net influx at vertex $x$. 
Condition $(ii)$ states that the influxes are absolutely summable (this can be thought of as a relaxation of Kirchoff's law) and that 
there is a nonzero net influx into the network. 
Condition $(iii)$ says that the Thompson energy of the flow is finite. 
We are going to prove Theorem~\ref{thm:Resistance} by constructing such a flow on the graph $\mathcal{I}$.

Denote by $S^d$ the $d$-dimensional Euclidean unit sphere. Given $v \in S^d$ and $\varepsilon \in (0,1)$, 
we define the  graph $\mathcal{I}(v,\varepsilon)$ by
\begin{equation}\label{def_eq_epsilon_tube_parabola}
\mathcal{I}(v,\varepsilon) = \mathcal{I} \cap \bigcup_{n=1}^{\infty} B(nv, n^{\varepsilon}).
\end{equation}
The set $\bigcup_{n=1}^{\infty} B(nv, n^{\varepsilon} )$ is roughly shaped like a paraboloid with an axis parallel to $v$.
We denote by $\mathcal C_\infty(v,\varepsilon)$ the maximal subgraph of $\mathcal I(v,\varepsilon)$ in which every connected component is infinite. 
(If $\mathcal I(v,\varepsilon)$ does not contain an infinite connected component, we set $\mathcal C_\infty(v,\varepsilon) = \emptyset$.)

\begin{lemma}\label{l:tube}
For any $u>0$ and $0< \varepsilon<1$, we have
\begin{equation}\label{every_direction_infinite_component}
\mathbb P\left[\bigcup_{M\geq 1} \left\{\forall \, v \in S^d \,: \; B(M)\cap\mathcal C_\infty(v,\varepsilon) \neq \emptyset\right\}\right] = 1 .\
\end{equation}
\end{lemma}

\begin{proof}
For any $z \in \Z^d$, define the events
\begin{equation*}
A_z= \left\{ \forall \, x,y \in \mathcal{I} \cap B(z, \frac 14 |z|^\varepsilon) \, : \; 
 x \stackrel{ \mathcal{I} \cap B(z, \frac 12 |z|^{\varepsilon}) }{\longleftrightarrow} y  \right\}, 
\qquad  B_z= \left\{  \mathcal{I} \cap  B(z, \frac 18 |z|^\varepsilon ) \neq \emptyset \right\}.
\end{equation*}
It follows from the Borel-Cantelli lemma and Proposition~\ref{prop:Connectivity} that $\PP(\liminf_{z \in \Z^d} A_z \cap B_z)=1$, i.e.,
almost surely the number of $z \in \Z^d$ for which $A_z \cap B_z$ does not occur is finite. 

Note that there exists an integer $m$ such that for every $v\in S^d$, 
there exists a $\Z^d$-valued sequence $(z_i)_{i=1}^\infty$ with $z_1\in B(m)$ and $|z_i|\to\infty$, such that for all $i$, 
(a) $B(z_i,\frac 12 |z_i|^\varepsilon) \subset \bigcup_{n=1}^{\infty} B(nv, n^{\varepsilon} )$ and 
(b) $B(z_i,\frac 18 |z_i|^\varepsilon)$ and $B(z_{i+1},\frac 18 |z_{i+1}|^\varepsilon)$ are subsets of $B(z_{i+1},\frac 14 |z_{i+1}|^\varepsilon)$.
Indeed, one can take, for example, a discrete approximation of the $\R^d$-valued sequence $((i+i_0) v)_{i=1}^\infty$ 
for large enough $i_0$. (Note that $i_0$ can be chosen independent of $v$.) 

Let $M$ be an almost surely finite random variable such that, for all $v\in S^d$ and $i \geq M$, the events  $A_{z_i}$ and $B_{z_i}$ hold. 
Then, for all $i\geq M$, $\mathcal{I} \cap B(z_i, \frac 18 |z_i|^\varepsilon) \neq \emptyset$, 
and  every vertex in $\mathcal{I} \cap B(z_i, \frac 18 |z_i|^\varepsilon)$ is connected to 
every vertex in $\mathcal{I} \cap B(z_{i+1},  \frac 18 |z_{i+1}|^\varepsilon)$ by a path in $\mathcal{I}(v,\varepsilon)$.
This implies \eqref{every_direction_infinite_component}. 
\end{proof}

\begin{definition}\label{def:flowpath}
It follows from Lemma \ref{l:tube} that we can almost surely assign (in a measurable way) to every $v \in S^d$ a (random) simple nearest-neighbor path $w_v = (w_v(n))_{n=0}^\infty$ in the graph $\mathcal I(v,\varepsilon)$.
In particular, for all $n\neq m \in \N$, $w_v(n) \neq w_v(m)$, and $\lim_{n \to \infty} |w_v(n)|=\infty$. 
\end{definition}

Our construction of the flow $u$ with finite energy is analogous to the proof of P\'olya's theorem in \cite{lyons_peres_book}.
For every $v \in S^d$ define $(u_v(x,y))_{x,y \in \mathcal{I}}$ to be the unit flow that goes from $w_v(0)$ to $\infty$ along the simple path
$w_v$, more precisely let $u_v(x,y)= - u_v(y,x) = 1$ if $w_v(n)=x$ and $w_v(n+1)=y$ for some $n \in \N$ and, otherwise, let $u_v(x,y)=0$.
With this definition we have
\begin{equation}\label{flow_on_path_with_one_source}
\sum_{y \in \mathcal{I}} u_v(x,y) = \mathds{1}[ x= w_v(0)].
\end{equation}
Note that for any $v \in S^d$ the flow $u_v$ satisfies $(i)$ and $(ii)$ of Lemma \ref{l:lyons}, but fails to satisfy $(iii)$.

We define the flow $u$ as the average of the flows $u_v$ with respect to $v$, more precisely let 
\[u(x,y):=\int_{S^d} u_v(x,y) \, \mathrm{d}\lambda(v)\] where $\lambda$ is the Haar measure on $S^d$ normalized to be a probability measure.

Now we check that $u$ is a flow with finite energy on $\mathcal{I}$, i.e., the conditions of Lemma \ref{l:lyons} hold.
The function $u$ inherits property $(i)$ from the flows $u_v$. 
From \eqref{flow_on_path_with_one_source} it readily follows that we have $\sum_{y \in \mathcal{I}} u(x,y) \geq 0$ for all $x \in \mathcal{I}$ and  that
$\sum_{x \in \mathcal{I}}  \left(\sum_{y \in \mathcal{I}} u(x,y)\right) = 1$ holds, from which $(ii)$ follows.
It only remains to show that the energy of $u$ is finite, i.e., $(iii)$ holds. 
Note that $u(x,y)\neq 0$ only if $|x-y|=1$. 
We have
\begin{equation}\label{ray_parabula_tube_bound}
 |u(x,y)| \leq  \int_{S^d} |u_v(x,y)| \, \mathrm{d}\lambda(v) \leq 
\int_{S^d} \mathds{1}[ \, x  \in \bigcup_{n=1}^{\infty} B(nv, n^{\varepsilon} ) \, ] \, \mathrm{d}\lambda(v)
\leq 
C \frac{\left( |x|^{\varepsilon} \right)^{d-1}}{|x|^{d-1}}.
\end{equation}
Now choose $0<\varepsilon<\frac14$ in Definition~\ref{def:flowpath}. The corresponding flow $u$ has finite energy: 
\begin{multline*}
\sum_{x \in \Z^d}  \sum_{y \in \Z^d} u(x,y)^2 \leq
\sum_{n=1}^{\infty} \; \sum_{ |x|=n} \; \sum_{  |x-y|=1} u(x,y)^2 \stackrel{\eqref{ray_parabula_tube_bound}} {\leq} 
C\sum_{n=1}^{\infty} \; \sum_{ |x|=n} n^{ 2(\varepsilon -1)(d-1)} 
\stackrel{d \geq 3}{\leq} C  \sum_{n=1}^{\infty} n^{4 \varepsilon -2}
< \infty .\
\end{multline*}
Therefore, the flow $u$ satisfies the conditions of Lemma~\ref{l:lyons}, which proves Theorem~\ref{thm:Resistance}. 
\qed

\section{Proof of Proposition~\ref{prop:Connectivity}}\label{sec:connectivity}
\subsection{Bounds on the capacity of certain collections of random walk trajectories}\label{sec:capacity}
\noindent
The aim of this subsection is to prove Lemma~\ref{l:rwcaphighprob} (with $\Phi(\overline X_N, T)$ defined in \eqref{def:Phi}),  
which will be used in the proof of Lemma~\ref{l:intcaphighprob}. 

The following lemma is proved in \cite{RS} for $d\geq 5$, see Lemma~3 there. 
The cases $d=3,4$ can be proved similarly. 
Therefore, we state this lemma without proof. 
\begin{lemma}\label{l:gfestimate}
Let $(x_i)_{i\geq 1}$ be a sequence in $\Z^d$, and let $X_i$ be a sequence of independent simple random walks on $\Z^d$ with $X_i(0) = x_i$. 
Let 
\[
F(n,d) = \begin{cases} 
          n^{1/2} &\mbox{if } d=3,\\
          \log n &\mbox{if } d=4, \mbox{ and}\\
          1 &\mbox{if } d\geq 5.   
         \end{cases}
\]
Then for all positive integers $N$ and $n$, we have 
\begin{equation}\label{eq:gfestimate}
{\mathbf E}\left[\sum_{i,j=1}^N \sum_{s,t=n+1}^{2n} g\left(X_i(s),X_j(t)\right)\right]
\leq
C\left(NnF(n,d) + N^2 n^{3-d/2}\right) .\
\end{equation}
\qed
\end{lemma}

Let $(X_i(t)~:~t\geq 0)_{i\geq 1}$ be a sequence of nearest-neighbor trajectories on $\Z^d$, and 
$\overline X_N = (X_1,\ldots,X_N)$.  
For positive integers $N$ and $T$, we define the subset $\Phi(\overline X_N, T)$ of $\Z^d$ as
\begin{equation}\label{def:Phi}
\Phi(\overline X_N, T) 
=
\bigcup_{i=1}^N \left\{X_i(t)~:~1\leq t\leq T\right\} .\
\end{equation}
\begin{lemma}\label{l:rwcap}
Let $X_i$ be a sequence of independent simple random walks on $\Z^d$ with $X_i(0) = x_i$. 
There exists a positive constant $c$ such that for any sequence $(x_i)_{i\geq 1}\subset\Z^d$ and for all positive integers $N$ and $T$, 
\begin{equation}\label{eq:rwcap}
\mathrm{cap}\left(\Phi (\overline X_N, T)\right) \leq \frac{NT}{g(0)}, \quad\quad\mbox{and}\quad\quad
\mathbf E \mathrm{cap}\left(\Phi(\overline X_N,T)\right) \geq c \min\left(\frac{N T}{F(T,d)}, T^{\frac{d-2}{2}}\right) ,\
\end{equation}
where the function $F$ is defined in Lemma~\ref{l:gfestimate}. 
\end{lemma}
\begin{remark}
Heuristically, $\mathbf E \mathrm{cap}\left(\Phi(\overline X_N,T)\right)$ is at least $\frac{cNT}{F(T,d)}$ when the random walks in $\overline X_N$ 
are well separated, and at least $\mathrm{cap}(B(cT^{1/2})) \stackrel{\eqref{eq:capball}}\geq c T^{\frac{d-2}{2}}$ 
when the set $\Phi(\overline X_N,T)$ saturates the ball $B(cT^{1/2})$. 
\end{remark}
\begin{proof}
The proof of this lemma is similar to the proof of Lemma~4 in \cite{RS}, so we give only a sketch here. 
(Note that the definition of $\Phi(\overline X_N,R)$ in \cite{RS} is different from the one in \eqref{def:Phi}.)

The upper bound on the capacity of $\Phi (\overline X_N, T)$ follows from \eqref{cap:subadditivity} and \eqref{cap:point}. 
Let $n = \lfloor T/2\rfloor$. By the definition of the capacity \eqref{eq:defcap} and the Jensen inequality, 
\[
\mathbf E \mathrm{cap}\left(\Phi(\overline X_N,T)\right)
\geq
N^2n^2
\left(\mathbf E\left[\sum_{i,j=1}^N \sum_{s,t=n+1}^{2n} g(X_i(s),X_j(t))\right]\right)^{-1} .\
\]
The lower bound in \eqref{eq:rwcap} now follows from Lemma~\ref{l:gfestimate}. 
\end{proof}

As a corollary of Lemma~\ref{l:rwcap} we obtain the following lemma. 
\begin{lemma}\label{l:rwcapprob}
Let $X_i$ be a sequence of independent simple random walks on $\Z^d$ with $X_i(0) = x_i$. 
There exists a positive constant $c$ such that for any sequence $(x_i)_{i\geq 1}\subset\Z^d$ and for all positive integers $N$ and $T$, 
\begin{equation}\label{eq:rwcapprob}
\mathbf P \left[\mathrm{cap}\left(\Phi (\overline X_N,T)\right) \geq c \min\left(\frac{N T}{F(T,d)}, T^{\frac{d-2}{2}}\right) \right]
\geq 
\frac{c}{(\log T)^2} .\
\end{equation}
\end{lemma}
\begin{proof}
Remember the Paley-Zygmund inequality \cite{PZ}: 
Let $\xi$ be a non-negative random variable with finite second moment. For any $\theta\in(0,1)$, 
$\mathrm P[\xi\geq \theta \mathrm E\xi] \geq (1-\theta)^2\left[\mathrm E \xi\right]^2/\mathrm E[\xi^2]$. 

We first consider the case $d=3$. 
Note that in this case, 
$\min\left(\frac{N T}{F(T,d)}, T^{\frac{d-2}{2}}\right) = T^{1/2}$. 
Since $\Phi(\overline X_N,T) \supseteq \Phi(\overline X_1,T)$, by \eqref{cap:monotonicity}, it suffices to show that 
\[
\mathbf P \left[\mathrm{cap}\left(\Phi (\overline X_1,T)\right) \geq c T^{1/2} \right] \geq c .\
\]
It follows from \eqref{eq:rwcap} that 
$\mathbf E \mathrm{cap}\left(\Phi(\overline X_1,T)\right) \geq cT^{1/2}$. 
On the other hand, the set $\Phi(\overline X_1,T)$ is contained in $B(X_1(0),M)$, 
where $M = \max\{|X_1(t)-X_1(0)|~:~1\leq t\leq T\}$.  
Therefore, by \eqref{cap:monotonicity} and \eqref{eq:capball}, 
$\mathrm{cap}\left(\Phi(\overline X_1,T)\right) \leq CM$. 
In particular, since $\mathbf E M^2 \leq CT$, we have 
$\mathbf E \left[\mathrm{cap}\left(\Phi(\overline X_1,T)\right)^2\right] \leq C T$. 
The result now follows from the Paley-Zygmund inequality. 

Let $d\geq 4$. An application of the Paley-Zygmund inequality and \eqref{eq:rwcap} gives 
\begin{equation}\label{eq:rwcapPZ}
\mathbf P \left[\mathrm{cap}\left(\Phi(\overline X_N,T)\right) \geq c \min\left(\frac{N T}{F(T,d)}, T^{\frac{d-2}{2}}\right) \right]
\geq 
\frac{c\min\left(\frac{N T}{F(T,d)}, T^{\frac{d-2}{2}}\right)^2}{N^2T^2} .\
\end{equation}
We distinguish two cases.  If $\min\left(\frac{N T}{F(T,d)}, T^{\frac{d-2}{2}}\right)  = \frac{N T}{F(T,d)}$, 
the result immediately follows from \eqref{eq:rwcapPZ} and the definition of $F$. 
Otherwise, if $\min\left(\frac{N T}{F(T,d)}, T^{\frac{d-2}{2}}\right)  = T^{\frac{d-2}{2}}$, 
we take $N'\leq N$ in $\left[ T^{\frac{d-4}{2}}F(T,d),  2T^{\frac{d-4}{2}}F(T,d)\right]$. 
Such a choice is possible, since $1\leq T^{\frac{d-4}{2}}F(T,d) \leq N$. 
The result then follows from \eqref{eq:rwcapPZ} by observing that 
$\mathrm{cap}\left(\Phi(\overline X_N,T)\right) \geq \mathrm{cap}\left(\Phi(\overline X_{N'},T)\right)$ by \eqref{cap:monotonicity}.
\end{proof}

In the next lemma we show that the capacity of $\Phi(\overline X_N,T)$ is large with high probability. 
\begin{lemma}\label{l:rwcaphighprob}
Let $\varepsilon\in(0,1)$. 
Let $X_i$ be a sequence of independent simple random walks on $\Z^d$ with $X_i(0) = x_i$. 
There exists a positive constant $c$ such that for any sequence $(x_i)_{i\geq 1}\subset\Z^d$ and for all positive integers $N$ and $T$, 
\begin{equation}\label{eq:rwcaphighprob}
\mathbf P \left[\mathrm{cap}\left(\Phi(\overline X_N,T)\right) \geq c \min\left(N T^{\frac{1-\varepsilon}{2}}, T^{\frac{(d-2)(1-\varepsilon)}{2}}\right) \right]
\geq 
1 - \exp\left(- cT^{\varepsilon/2}\right) .\
\end{equation}
\end{lemma}
\begin{proof}
For positive integers $N$, $\widetilde T$ and $k$, we define the subset $\Phi_k(\overline X_N, \widetilde T)$ of $\Z^d$ by 
\[
\Phi_k(\overline X_N, \widetilde T) 
=
\bigcup_{i=1}^N \left\{X_i(t)~:~(k-1)\widetilde T + 1\leq t\leq k\widetilde T\right\} .\
\]
It follows from the Markov property, \eqref{eq:rwcapprob}, and the definition of the function $F$, that 
\[
\mathbf P \left[\mathrm{cap}\left(\Phi_k (\overline X_N,\widetilde T)\right) \geq c \min\left(N \widetilde T^{1/2}, \widetilde T^{\frac{d-2}{2}}\right) 
~\mid~X_i(t),~i\in\{1,\ldots,N\},~t\leq (k-1)\widetilde T\right]
\geq 
\frac{c}{(\log \widetilde T)^2} .\
\]
Therefore, for any $\delta>0$, 
\[
\mathbf P \left[\mathrm{cap}\left(\bigcup_{k=1}^{\lfloor \widetilde T^{\delta}\rfloor }\Phi_k(\overline X_N,\widetilde T)\right) 
\geq c \min\left(N \widetilde T^{1/2}, \widetilde T^{\frac{d-2}{2}}\right) \right]
\geq 
1 - \left[1 - \frac{c}{(\log \widetilde T)^2}\right]^{\lfloor \widetilde T^{\delta}\rfloor} 
\geq
1 - \exp\left(-c\widetilde T^{\delta/2}\right) .\
\]
The result follows by observing that 
$\bigcup_{k=1}^{\lfloor\widetilde T^{\delta}\rfloor}\Phi_k(\overline X_N,\widetilde T) \subseteq \Phi(\overline X_N,\lfloor\widetilde T^{1+\delta}\rfloor)$, 
and by taking $\varepsilon = \delta/(1+\delta)$. 
\end{proof}

\subsection{Bounds on the capacity of certain subsets of random interlacement}\label{sec:intcapacity}
\noindent

The aim of this subsection is to prove Lemma~\ref{l:capC}, which states that 
with high probability for $x\in \mathcal I$, the connected component of $\mathcal I\cap B(x,R)$ that contains $x$ has large capacity. 
We prove the statement by constructing explicitly for $x\in\mathcal I$ a connected subset of $\mathcal I\cap B(x,R)$ of large capacity that contains $x$. 
In this construction, we exploit property (4) of $\mathrm{Pois}(u,W^*)$, which allows to describe $\mathcal I$ as 
the union of independent identically distributed random interlacement graphs $\mathcal I_1,\ldots,\mathcal I_{d-2}$. 
We begin with auxiliary lemmas. 

Let $A$ be a finite set of vertices in $\Z^d$. 
Let $\mu$ be a random point measure with distribution $\mathrm{Pois}(u,W^*)$, and $\mu_A$ its restriction to the set of trajectories from $W^*$ that intersect $A$. 
We can write the measure $\mu_A$ as $\sum_{i=1}^{N_A} \delta_{\pi^*(X_i)}$ (recall the notation from Section~\ref{sec:model}), where 
$N_A = |\mathrm{Supp}(\mu_A)|$, and $X_1,\ldots,X_{N_A}$ are doubly-infinite trajectories from $W$ parametrized in such a way that 
$X_i(0)\in A$ and $X_i(t)\notin A$ for all $t<0$ and for all $i\in\{1,\ldots,N_A\}$.
We define the set $\Psi(\mu,A,T)$ as 
\begin{equation}\label{def:Psi}
\Psi(\mu,A,T) = \bigcup_{i=1}^{N_A} \left\{X_i(t)~:~1\leq t\leq T\right\} .\
\end{equation}
\begin{lemma}\label{l:intcaphighprob}
Let $\varepsilon\in(0,1)$. 
Let $\mu$ be a Poisson point measure with distribution $\mathrm{Pois}(u,W^*)$, then 
for all finite subsets $A$ of $\Z^d$ and for all positive integers $T$, one has
\begin{equation}\label{eq:intcaphighprob}
\mathbb P \left[\mathrm{cap}\left(\Psi(\mu,A,T)\right) \geq c \min\left(\mathrm{cap}(A) T^{\frac{1-\varepsilon}{2}}, T^{\frac{(d-2)(1-\varepsilon)}{2}}\right) \right]
\geq 
1 - Ce^{- c\min\left(T^{\varepsilon/2},~\mathrm{cap}(A)\right)} .\
\end{equation}
\end{lemma}
\begin{proof}
It follows from property (1) of $\mathrm{Pois}(u,W^*)$ that $N_A$ has the Poisson distribution with parameter $u\mathrm{cap}(A)$. 
Therefore, by \eqref{eq:Poisson}, 
$\mathbb P\left[N_A \geq c\mathrm{cap}(A)\right] \geq 1 - Ce^{-c\mathrm{cap}(A)}$. 
Properties (2) and (3) of $\mathrm{Pois}(u,W^*)$ imply that given $N_A$, 
the forward trajectories $X_1,\ldots,X_{N_A}$ are distributed as independent simple random walks. 
Therefore, Lemma~\ref{l:rwcaphighprob} applies, giving that 
\[
\mathbb P \left[\mathrm{cap}\left(\Psi(\mu,A,T)\right) \geq c \min\left(N_A T^{\frac{1-\varepsilon}{2}}, T^{\frac{(d-2)(1-\varepsilon)}{2}}\right) \right]
\geq 
1 - e^{- cT^{\varepsilon/2}} .\
\]
The result follows.
\end{proof}

Let $X$ be a simple random walk on $\mathbb Z^d$ with $X(0) = x$. 
Let $\mu^{(2)}, \mu^{(3)}, \ldots$ be independent random point measures with distribution $\mathrm{Pois}(u,W^*)$ (the parameter $u$ is fixed here), 
which are also independent of $X$. 
We denote by $\mathbb P_x$ the joint law of $X$ and $\mu^{(i)}$'s. 
Let $T$ be a positive integer. 
We define the following sequence of random subsets of $\Z^d$: 
\begin{equation}\label{eq:u1}
U^{(1)}(x,T) = \left\{X(t)~:~1\leq t\leq T\right\} ,\
\end{equation}
and for $s\geq 2$ (see \eqref{def:Psi} for notation), 
\begin{equation}\label{eq:u2}
U^{(s)}(x,T) = \Psi \left(\mu^{(s)}, U^{(s-1)}(x,T), T\right) .\
\end{equation}
Note that for each $s \geq 1$, $\bigcup_{i=1}^s U^{(i)}(x,T)$ is a connected subset of $\Z^d$. 
In the next lemma, we show that for any $\gamma>0$, with high probability, 
the set $\bigcup_{i=1}^s U^{(i)}(x,T)$ is a subset of $B(x,sT^{(1+\gamma)/2})$.  
\begin{lemma}\label{l:ubounds}
Let $\gamma\in(0,1)$. There exist $c=c(u,d,s)>0$ and $C = C(u,d,s)<\infty$ such that 
\begin{equation}\label{eq:ubounds}
\mathbb P_x \left[\bigcup_{i=1}^sU^{(i)}(x,T) \subseteq B(x,sT^{(1+\gamma)/2})\right] \geq 1 - Ce^{-cT^{\gamma}} .\
\end{equation}
\end{lemma}
\begin{proof}
We denote the event $\{\bigcup_{i=1}^sU^{(i)}(x,T) \subseteq B(x,sT^{(1+ \gamma)/2})\}$ by $D_s$, and its complement by $D_s^c$. 
If $s=1$, \eqref{eq:ubounds} follows from Hoeffding's inequality: $\mathbb P_x \left[D_1^c\right] \leq 2de^{-T^{\gamma}/8}$. 
Assume that \eqref{eq:ubounds} is proved for $s'<s$. 
Then 
\[
\mathbb P_x \left[D_s^c\right] \leq \mathbb P \left[D_{s-1}^c\right] + 
\mathbb P_x \left[D_s^c, D_{s-1}\right] ,\ 
\]
and it remains to show that $\mathbb P_x \left[D_s^c, D_{s-1}\right] \leq Ce^{-cT^{\gamma}}$. 
Note that if $D_{s-1}$ occurs, 
\[
|\mathrm{Supp}(\mu^{(s)}_{U^{(s-1)}(x,T)})| \leq |\mathrm{Supp}(\mu^{(s)}_{B(x,(s-1)T^{(1+\gamma)/2})})| .\ 
\]
(See Section~\ref{sec:model} for the notation.) 
Let us denote the right hand side of the above inequality by $N$. 
It follows from property (1) of $\mathrm{Pois}(u,W^*)$ that $N$ has the Poisson distribution with parameter 
$u\mathrm{cap}\left(B(x,(s-1)T^{(1+\gamma)/2})\right)$.  
In particular, using \eqref{eq:capball} and \eqref{eq:Poisson}, we obtain that
\[
{\mathbb P}_x\left[N\geq CT^{(d-2)(1+\gamma)/2}\right] \leq Ce^{-cT^{(1+\gamma)/2}} \leq Ce^{-cT^\gamma} .\
\]
On the other hand, properties (2) and (3) of $\mathrm{Pois}(u,W^*)$ imply that
\[
\mathbb P_x \left[D_s^c, D_{s-1}\right] \leq
{\mathbb P}_x \left[N\geq CT^{(d-2)(1+\gamma)/2}\right] 
+ 
CT^{(d-2)(1+\gamma)/2}\mathbb P_x \left[D_1^c\right] 
\leq Ce^{-cT^{\gamma}} .\
\]
This completes the proof of the lemma. 
\end{proof}

The next lemma follows immediately from Lemma~\ref{l:intcaphighprob} and the definition of $U^{(s)}(x,T)$. 
\begin{lemma}\label{l:capuhighprob}
Let $\varepsilon\in(0,1/2)$. For any positive integer $s$, there exist $c = c(d,u,s)>0$ and $C = C(d,u,s)<\infty$ such that for all positive integers $T$, 
\begin{equation}\label{eq:capuhighprob}
\mathbb P_x \left[\bigcap_{i=1}^s\left\{\mathrm{cap}\left(U^{(i)}(x,T)\right) \geq c \min\left(T^{\frac{i(1-\varepsilon)}{2}}, T^{\frac{(d-2)(1-\varepsilon)}{2}}\right)\right\} \right]
\geq 
1 - C\exp\left(- cT^{\varepsilon/2}\right) .\
\end{equation}
\end{lemma}
\begin{proof}
The case $s=1$ follows from \eqref{eq:u1} and Lemma~\ref{l:rwcaphighprob}. Let $s\geq 2$. 
Note that, for $c\in(0,1)$, the event in \eqref{eq:capuhighprob} with $c^s$ in place of $c$
is implied by the event 
\[
\bigcap_{i=1}^s\left\{\mathrm{cap}\left(U^{(i)}(x,T)\right) \geq c \min\left(\mathrm{cap}\left(U^{(i-1)}(x,T)\right)T^{\frac{1-\varepsilon}{2}}, T^{\frac{(d-2)(1-\varepsilon)}{2}}\right)\right\} ,\
\]
where we set by convention $\mathrm{cap}\left(U^{(0)}(x,T)\right) = 1$. 
The result now follows (by induction in $s$) from \eqref{eq:capuhighprob} for $s=1$, Lemma~\ref{l:intcaphighprob} and the definition of $U^{(s)}(x,T)$, see \eqref{eq:u2}. 
\end{proof}

\begin{corollary}
It follows from Lemmas \ref{l:ubounds} and \ref{l:capuhighprob} that for any $\varepsilon\in(0,1/2)$, 
\begin{equation}\label{eq:uhighprob}
\mathbb P_x \left[
\bigcup_{i=1}^{d-2}U^{(i)}(x,T) \subseteq B(x,(d-2)T^{(1+\varepsilon)/2}),~~
\mathrm{cap}\left(U^{(d-2)}(x,T)\right) \geq c T^{\frac{(d-2)(1-\varepsilon)}{2}}
\right] 
\geq 
1 - Ce^{-cT^{\varepsilon/2}} .\
\end{equation}
In particular on the event in \eqref{eq:uhighprob}, $\bigcup_{i=1}^{d-2}U^{(i)}(x,T)$ is a connected subset of $B(x,(d-2)T^{(1+\varepsilon)/2})$. 
\end{corollary}

Let $\mu$ be a Poisson point measure with distribution $\mathrm{Pois}(u,W^*)$, and let $\mathcal I$ be the corresponding random interlacement graph at level $u$.  
(See Section~\ref{sec:model} for the definition.) 
For $x\in\mathcal I$, let $\mathcal C(x,R)$ be the connected component of $\mathcal I\cap B(x,R)$ that contains $x$. 
We define $\mathcal C(x,R)$ as an empty set for $x\notin \mathcal I$. 
In the next lemma we show that for $x\in\mathcal I$, the capacity of $\mathcal C(x,R)$ is large enough with high probability. 
\begin{lemma}\label{l:capC}
For all $\varepsilon\in(0,2/3)$, $R>0$, and $x\in \Z^d$, we have  
\begin{equation}\label{eq:capC}
\mathbb P\left[x\in\mathcal I,~ \mathrm{cap}\left(\mathcal C(x,R)\right) < cR^{(d-2)(1-\varepsilon)} \right] \leq Ce^{-cR^{\varepsilon/2}} .\
\end{equation}
\end{lemma}

\begin{proof}
Let $\mu^{(1)},\ldots,\mu^{(d-2)}$ be independent Poisson point measures with distribution $\mathrm{Pois}(\frac{u}{d-2},W^*)$. 
Let $\mathbb P$ be the joint law of $\mu^{(i)}$. 
By property (4) of $\mathrm{Pois}(u,W^*)$, the measure $\mu$ has the same law as $\sum_{i=1}^{d-2}\mu^{(i)}$. 
Therefore, we may assume that $\mu = \sum_{i=1}^{d-2}\mu^{(i)}$. 
In particular, the random interlacement graph $\mathcal I = \mathcal I(\mu)$ equals 
$\bigcup_{i=1}^{d-2}\mathcal I^{(i)}$, where $\mathcal I^{(i)} = \mathcal I(\mu^{(i)})$ are independent random interlacement graphs at level $\frac{u}{d-2}$, 
and the vertices and edges of $\mathcal I$ are the ones of $\Z^d$ that are traversed by at least one of the random walks from 
$\bigcup_{i=1}^{d-2}\mathrm{Supp}(\mu^{(i)})$. 

It follows from \eqref{eq:uhighprob} (with $T = \widetilde R^2$) and property (3) of $\mathrm{Pois}(\frac{u}{d-2},W^*)$ that for any 
$\delta\in(0,1/2)$, $\widetilde R>0$, $x\in B(\widetilde R)$ and $i\in\{1,\ldots,d-2\}$, 
\[
\mathbb P\left[x\in\mathcal I^{(i)},~ \mathrm{cap}\left(\mathcal C(x,\widetilde R^{1+\delta})\right) < c\widetilde R^{(d-2)(1-\delta)} \right] 
\leq Ce^{-c\widetilde R^{\delta}} .\
\]
The result follows by taking $\delta = \varepsilon/(2-\varepsilon)$. 
\end{proof}

\subsection{Proof of Proposition~\ref{prop:Connectivity}}
\noindent

Proposition~\ref{prop:Connectivity} will follow from Lemma~\ref{l:twopoints}, 
which states that with high probability, 
any two vertices from $\mathcal I\cap B(R)$ are connected by a path in $\mathcal I\cap B(CR)$ for large enough $C$. 
This result will follow from \eqref{eq:capC}, \eqref{eq:connectedsets} and property (4) of $\mathrm{Pois}(u,W^*)$. 

We begin with auxiliary lemmas. Lemma~\ref{l:hitting1rw} is standard, so we only give a sketch of the proof here. 
\begin{lemma}\label{l:hitting1rw}
Let $A$ be a subset of $B(R)$. 
Let $X$ be a simple random walk on $\Z^d$ with $X(0)=x\in B(R)$. 
Let $H_A$ be the entrance time of $X$ in $A$, and $T_{B(r)}$ the exit time of $X$ from $B(r)$. 
Then there exist $c>0$ and $C<\infty$ such that for all $R>0$ and $x\in B(R)$, 
\[
P_x \left[H_A < T_{B(CR)}\right] \geq c R^{2-d}\mathrm{cap}(A) .\
\]
\end{lemma}
\begin{proof}
We use the identity (\ref{eq:hittingformula}).   
Since $A$ is a subset of $B(R)$, the inequality \eqref{eq:gfbounds} implies that, for any $y\in A$ and 
$x\in B(R)$, $g(x,y) \geq c_g (2R)^{2-d}$. 
By \eqref{eq:defcap2} and \eqref{eq:hittingformula}, 
$P_x\left[H_{A} < \infty\right] \geq c_g (2R)^{2-d} \mathrm{cap}(A)$. 

On the other hand, for $y\in A$ and $z\notin B(CR)$, the inequality \eqref{eq:gfbounds} gives
$g(z,y) \leq C_g ((C-1)R)^{2-d}$. 
Therefore, by \eqref{eq:defcap2}, \eqref{eq:hittingformula} and the strong Markov property of $X$ applied at time $T_{B(CR)}$, we have
$P_x\left[T_{B(CR)}<H_{A} < \infty\right] \leq C_g((C-1)R)^{2-d}\mathrm{cap}(A)$. 
The result follows by taking $C$ large enough. 
\end{proof}
\begin{lemma}\label{l:connectedsets}
There exist $c>0$ and $C<\infty$ such that for all $R>0$ and for all subsets $U$ and $V$ of $B(R)$, we have 
\begin{equation}\label{eq:connectedsets}
\mathbb P \left[U \stackrel{\mathcal I\cap B(CR)}\longleftrightarrow V\right] 
\geq 
1 - C\exp\left(-cR^{2-d}\mathrm{cap}(U)\mathrm{cap}(V)\right) .\
\end{equation}
\end{lemma}
\begin{proof}
Let $\mu$ be a random point measure with distribution $\mathrm{Pois}(u,W^*)$. 
Remember from Section~\ref{sec:model} that 
$\mu_{U} = \sum_{i=1}^{N_{U}} \delta_{\pi^*(X_i)}$, where 
$N_{U} = |\mathrm{Supp}(\mu_U)|$ 
and $X_1,\ldots,X_{N_{U}}$ are doubly-infinite trajectories from $W$ 
parametrized in such a way that $X_i(0)\in U$ and $X_i(t)\notin U$ for all $t<0$ and for all $i\in\{1,\ldots,N_{U}\}$. 

It follows from property (1) of $\mathrm{Pois}(u,W^*)$ and \eqref{eq:Poisson} that 
$\mathbb P\left(N_U \geq c\mathrm{cap}(U)\right) \geq 1 - Ce^{-c\mathrm{cap}(U)}$. 
Therefore, by Lemma~\ref{l:hitting1rw} and properties (2) and (3) of $\mathrm{Pois}(u,W^*)$, we have 
\begin{eqnarray*}
\mathbb P\left[U \stackrel{\mathcal I\cap B(CR)}\longleftrightarrow V\right] 
&\geq 
&1 - \mathbb P\left(N_U < c\mathrm{cap}(U)\right)
- \left(1 - cR^{2-d}\mathrm{cap}(V)\right)^{c\mathrm{cap}(U)}\\
&\geq 
&1 - C\exp\left(- cR^{2-d}\mathrm{cap}(U) \mathrm{cap}(V)\right) .\
\end{eqnarray*}
In these inequalities we also used the fact that $\mathrm{cap}(V) \leq CR^{d-2}$, which follows from \eqref{cap:monotonicity} and \eqref{eq:capball}. 
The proof is complete. 
\end{proof}
As a corollary of Lemmas \ref{l:capC} and \ref{l:connectedsets} we get the following lemma. 
\begin{lemma}\label{l:twopoints}
There exist $c>0$ and $C<\infty$ such that for all $R>0$ and $x,y\in B(R)$, we have 
\begin{equation}\label{eq:twopoints}
\mathbb P\left[x,y\in\mathcal I, \left\{x\stackrel{\mathcal I\cap B(CR)}\longleftrightarrow y\right\}^c\right] \leq C\exp\left(-cR^{1/6}\right) .\
\end{equation}
\end{lemma}
\begin{proof}
Let $\mu$ be a Poisson point measure with distribution $\mathrm{Pois}(u,W^*)$, 
and $\mu^{(1)}$, $\mu^{(2)}$ and $\mu^{(3)}$ be independent Poisson point measures with distribution $\mathrm{Pois}(u/3,W^*)$. 
Let $\mathbb P$ be the joint law of $\mu^{(i)}$. 
By property (4) of $\mathrm{Pois}(u,W^*)$, the measure $\mu$ has the same law as $\sum_{i=1}^{3}\mu^{(i)}$. 
Therefore, we may assume that $\mu = \sum_{i=1}^{3}\mu^{(i)}$, so that the random interlacement graph $\mathcal I = \mathcal I(\mu)$ equals 
$\bigcup_{i=1}^3\mathcal I^{(i)}$, where $\mathcal I^{(i)} = \mathcal I(\mu^{(i)})$ are independent random interlacement graphs at level $u/3$. 
In particular, the vertices and edges of $\mathcal I$ are the ones of $\Z^d$ that are traversed by at least one of the random walks from 
$\bigcup_{i=1}^{3}\mathrm{Supp}(\mu^{(i)})$. 

Let $\mathcal C^{(i)}(x,R)$ be the connected component of $x$ in $\mathcal I^{(i)}\cap B(x,R)$.  
In particular, $\mathcal C^{(i)}(x,R) \subseteq \mathcal C(x,R)$, but it is not true in general that 
$\bigcup_{i=1}^3\mathcal C^{(i)}(x,R) = \mathcal C(x,R)$. 
Since $R$ is fixed throughout the proof, we write $\mathcal C^{(i)}(x)$ for $\mathcal C^{(i)}(x,R)$. 
We have for $x,y\in B(R)$, 
\[
\mathbb P\left[x,y\in\mathcal I, \left\{x\stackrel{\mathcal I\cap B(CR)}\longleftrightarrow y\right\}^c\right]
\leq 
\sum_{i,j=1}^3 
\mathbb P\left[x\in\mathcal I^{(i)}, y\in\mathcal I^{(j)}, \left\{x\stackrel{\mathcal I\cap B(CR)}\longleftrightarrow y\right\}^c\right] .\
\]
For each $i,j\in\{1,2,3\}$, choose $k\in\{1,2,3\}$ which is different from $i$ and $j$. 
By construction, the set $\mathcal I^{(k)}$ is independent from $\mathcal I^{(i)}$ and $\mathcal I^{(j)}$. 
For each such $i$, $j$, and $k$, we obtain
\[
\mathbb P\left[x\in\mathcal I^{(i)}, y\in\mathcal I^{(j)}, \left\{x\stackrel{\mathcal I\cap B(CR)}\longleftrightarrow y\right\}^c\right] 
\leq 
\mathbb P\left[x\in\mathcal I^{(i)}, y\in\mathcal I^{(j)}, \left\{\mathcal C^{(i)}(x)\stackrel{\mathcal I^{(k)}\cap B(CR)}\longleftrightarrow \mathcal C^{(j)}(y)\right\}^c\right] .\
\]
We define the events $E_1\subseteq \{x\in\mathcal I^{(i)}\}$ and $E_2\subseteq \{y\in\mathcal I^{(j)}\}$ as 
\[
E_1 = \left\{\mathrm{cap}\left(\mathcal C^{(i)}(x)\right) > cR^{2(d-2)/3}\right\},~~\mbox{and}~~
E_2 = \left\{\mathrm{cap}\left(\mathcal C^{(j)}(y)\right) > cR^{2(d-2)/3}\right\} .\
\]
We denote the intersection $E_1\cap E_2$ by $E$. 
By Lemma~\ref{l:capC} (with $\varepsilon = 1/3$), we get 
\[
\mathbb P\left[\left\{x\in \mathcal I^{(i)}\right\}\setminus E_1 \right] 
+
\mathbb P\left[\left\{y\in \mathcal I^{(j)}\right\}\setminus E_2\right]
\leq
Ce^{-cR^{1/6}} .\
\]
Note that $\mathcal C^{(i)}(x)$ and $\mathcal C^{(j)}(y)$ are subsets of $B(2R)$. 
Therefore, it follows from Lemma~\ref{l:connectedsets} and the independence of $\mathcal I^{(k)}$ from $\mathcal I^{(i)}$ and $\mathcal I^{(j)}$, that 
\begin{eqnarray*}
\mathbb P\left[E \setminus 
\left\{\mathcal C^{(i)}(x)\stackrel{\mathcal I^{(k)}\cap B(CR)}\longleftrightarrow \mathcal C^{(j)}(y)\right\}\right] 
&\leq 
&C\mathbb E \left[\exp\left(- cR^{2-d}\mathrm{cap}\left(\mathcal C^{(i)}(x)\right)\mathrm{cap}\left(\mathcal C^{(j)}(y)\right)\right); E\right]\\
&\leq 
&C\exp\left(-cR^{(d-2)/3}\right) .\
\end{eqnarray*}
Putting the bounds together gives the result. 
\end{proof}

\begin{proof}[Proof of Proposition~\ref{prop:Connectivity}]
We will use a standard covering argument to derive \eqref{eq:Connectivity} from \eqref{eq:twopoints}. 
Take the constant $C$ from the statement of Lemma~\ref{l:twopoints}. 
It suffices to prove \eqref{eq:Connectivity} for $R \geq 2C$. 
Let $R' = \lfloor R/2C\rfloor$. 
For each $z\in \Z^d$, we define the events 
\[
A^{(1)}_z = \left\{\mathcal I\cap B(z,R')\neq \emptyset\right\}, \quad\quad
A^{(2)}_z = \bigcap_{x,y\in\mathcal I\cap B(z,2R')} \left\{x\stackrel{\mathcal I\cap B(z,R)}\longleftrightarrow y\right\} ,\
\]
and $A = \bigcap_{z\in B(R)} A^{(1)}_z\cap A^{(2)}_z$. 
It follows from property (1) of $\mathrm{Pois}(u,W^*)$ and \eqref{eq:capball} that 
\[
\mathbb P (A^{(1)}_z) 
= 1 - e^{- u\mathrm{cap}(B(z,R'))}
\geq 1 - e^{- cR} ,\
\]
and from Lemma~\ref{l:twopoints} that 
\[
\mathbb P(A^{(2)}_z) \geq 1 - Ce^{- cR^{1/6}} .\
\]
In particular, $\mathbb P (A) \geq 1 - C'\exp(- cR^{1/6})$. 
It remains to note that $A$ implies the event in \eqref{eq:Connectivity}. 
Indeed, for all $z,z'\in B(R)$ with $|z-z'| = 1$, $B(z,R')\cup B(z',R')\subseteq B(z,2R')$; thus 
if $A$ occurs then every vertex in the non-empty set $\mathcal I\cap B(z,R')$ is connected to every vertex in the non-empty set $\mathcal I\cap B(z',R')$ by a 
path in $\mathcal I\cap B(z,R) \subseteq B(2R)$.  
Since any two vertices in $B(R)$ are connected by a nearest-neigbor path in $B(R)$, $A$ implies the event in \eqref{eq:Connectivity}. 
The result follows. 
\end{proof}

\bigskip
\textbf{Acknowledgments.}  We would like to thank A.-S. Sznitman for suggesting this problem to us.

\end{document}